\documentclass [12pt,oneside]{amsart}
\usepackage{color}
\usepackage[margin=0.9in]{geometry}
\usepackage{amssymb,amsmath,amsthm,amsfonts}
\usepackage{mathrsfs}
\usepackage{enumerate}
\usepackage{setspace,bbm}
\usepackage{empheq}
\usepackage[all]{xy}
\usepackage{verbatim}
\usepackage[normalem]{ulem}
\usepackage{hyperref}
\usepackage{url}
\usepackage{todonotes}
\hypersetup{linkcolor=blue,citecolor=blue,filecolor=blue,urlcolor=blue} 

\numberwithin{equation}{section}
\numberwithin{figure}{section}

\newtheorem{theorem}{Theorem}[section]

\newtheorem{lemma}[theorem]{Lemma}

\newtheorem{corollary}[theorem]{Corollary}

\theoremstyle{definition}
\newtheorem{definition}[theorem]{Definition}
\newtheorem{example}[theorem]{Example}
\newtheorem{remark}[theorem]{Remark}

\definecolor{myblue}{rgb}{0.6, 0.9, 1}

\makeatletter

\newcommand{\Rmnum}[1]{\expandafter\@slowromancap\romannumeral #1@}
\makeatother

\definecolor{myblue}{rgb}{0.6, 0.9, 1}
\definecolor{mygreen}{rgb}{0,0,1}
\definecolor{purple}{rgb}{0.6,0.2,1}
\definecolor{orange}{rgb}{0.8,0,0.2}

\definecolor{purple}{rgb}{0.6,0.2,1}

\newcommand{\mP}{\mathbb{P}}

\newcommand{\mc}{\mathcal}

\newcommand{\Gal}{\operatorname{Gal}}
\newcommand{\Kbar}{\overline{K}}

\author{Jamie Juul, Holly Krieger, Nicole Looper, and Niki Myrto Mavraki}
\title{A Dynamical Shafarevich theorem for endomorphisms of $\mathbb{P}^N$}

\begin{document} 
	\begin{abstract} We prove a dynamical analogue of the Shafarevich conjecture for morphisms $f:\mathbb{P}_K^N\to\mathbb{P}_K^N$ of degree $d\ge2$, defined over a number field $K$. This extends previous work of Silverman \cite{Silverman:Shafarevich} and others in the case $N=1$.\end{abstract}
	\maketitle

\section{Introduction}

Let $K$ be a number field, let $S$ be a finite set of places of $K$ containing the archimedean places, and let $\mathcal{O}_S$ be the ring of $S$-integers of $K$. An abelian variety $A/K$ is said to have \emph{good reduction outside $S$} if there is a proper $\mathcal{O}_S$-group scheme with generic fiber $K$-isomorphic to $A$. The Shafarevich conjecture, proved by Faltings \cite{Faltings}, gives a finiteness statement for such abelian varieties. 

\begin{theorem}[Faltings]\cite{Faltings}\label{thm:Faltings} 
	Up to $K$-isomorphism, there are only finitely many principally polarized abelian varieties $A$ over $K$ of a given dimension which have good reduction outside $S$.
\end{theorem} 
	
	The goal of this paper is to provide a dynamical analogue of Theorem \ref{thm:Faltings} for morphisms $f:\mathbb{P}_K^N\to\mathbb{P}_K^N$ of degree $d\ge2$. 
Letting $\mathcal{O}_\mathfrak{p}$ denote the valuation ring of the local field $K_\mathfrak{p}$ for a prime $\mathfrak{p}$ of $K$, we say that such a morphism has \emph{good reduction outside $S$} if for all $\mathfrak{q}\notin S$, there is an $\mathcal{O}_\mathfrak{q}$-morphism $\mathbb{P}_{\mathcal{O}_\mathfrak{q}}^N\to\mathbb{P}_{\mathcal{O}_\mathfrak{q}}^N$ with generic fiber $\textup{PGL}_{N+1}(K_\mathfrak{q})$-conjugate to $f$ (cf. \cite[Remark 2.16]{Silverman:ADS}). The na\"ive dynamical analogue of the Shafarevich conjecture fails; for instance, when $N=1$, the set of monic degree $d\ge 2$ polynomials $f\in\mathcal{O}_S[x]$ consists of infinitely many $\textup{PGL}_{N+1}(K)$-conjugacy classes and all such polynomials have good reduction outside $S$. 
It is easy to generalize this construction to higher dimensions, so that the most obvious dynamical counterpart of the Shafarevich conjecture fails for all choices of $d,N,K$, and $S$. 
In order to recover an appropriate statement, one must impose more structure on the space of maps. 
As has been done under various guises in \cite{Petsche:criticallyseparable}, \cite{PetscheStout}, \cite{Silverman:Shafarevich}, \cite{SzpiroTucker}, and \cite{SzpiroWest}, we study pairs $(f,X)$ consisting of a map $f$ having good reduction outside $S$ and an appropriate finite $\Gal(\overline{K}/K)$-invariant set $X\subset\mathbb{P}^N(\overline{K})$ also having good reduction outside $S$ (see Definition \ref{good outside S}). Naturally, we would expect the set $X$ to be dynamically related to $f$ --- in particular, to have $X=Y\cup f(Y)$ for some $Y\subset\mathbb{P}^N(\overline{K})$ --- and for the set $X$ to have good reduction at all primes $\mathfrak{p}\notin S$, as in Definition \ref{good outside S}.

\begin{definition}\label{def:reductionofpoints} Let $x\in\mathbb{P}^N(K)$, let $\mathfrak{p}$ be a prime of $K$, and write $x=[x_0:\dots:x_N]$, where the $x_i$ are normalized so that they are $\mathfrak{p}$-integral for all $i$, and at least one $x_i$ is a $\mathfrak{p}$-adic unit. Then the \emph{reduction of $x$ modulo $\mathfrak{p}$} is the point $\widetilde{x}=[\widetilde{x_0}:\dots:\widetilde{x_N}]\in\mathbb{P}^N(k_{\mathfrak{p}})$, where $k_{\mathfrak{p}}$ is the residue field at $\mathfrak{p}$, and the $\widetilde{x_i}$ are the reductions of $x_i$ modulo $\mathfrak{p}$. If $X\subseteq\mathbb{P}^N(K)$, then $\widetilde{X}:=\{\widetilde{x}:x\in X\}$ is the \emph{reduction of $X$ modulo $\mathfrak{p}$}.
\end{definition}

Similarly, we define the reduction of a hyperplane as follows. 

\begin{definition}
	If $H_a: a_Nx_N+\cdots+a_0x_0=0$ is a hyperplane in $\mathbb{P}_K^N$ corresponding to $a=[a_0:\cdots:a_N]\in\mathbb{P}^N(K)$, then we write $$\widetilde{H_a}:=H_{\tilde{a}}: \tilde{a_N}x_N+\cdots+\tilde{a_0}x_0=0$$ for the reduction of $H_a$ in $\mathbb{P}^N_{k_{\mathfrak{p}}}$, where $\tilde{a}=[\tilde{a_0}:\cdots:\tilde{a_N}]$ is as in Definition \ref{def:reductionofpoints}.
\end{definition}

Good reduction of $X$ mod $\mathfrak{p}$ describes points in a certain kind of general position retaining that property upon reduction modulo the primes above $\mathfrak{p}$ in a field of definition for the points in $X$.
\begin{definition}\label{generalposition}
 Let $F$ be a field. We say that a finite set $X \subseteq\mathbb{P}^N(F)$ is in \emph{general linear position} if $|X|\ge N+1$ and no hyperplane contains $N+1$ points of $X$. 
\end{definition}

If a set $Z$ of $N$ points determines a unique hyperplane in $\mathbb{P}^N$, we denote this hyperplane by $H_Z$.  Note that if $X$ is in general linear position, then any subset $Z \subset X$ of cardinality $N$ determines a unique hyperplane $H_Z$. 

\begin{definition} \label{def:fieldofdefinition}
Let $X = \{ P_1, \dots, P_n \}$ and $G_K:= \Gal(\Kbar/K)$ then the \textit{field of definition} of the set $X$ over $K$, denoted $K(X) := K(P_1, \dots, P_n)$, is the subfield of $\overline{K}$ which is fixed by $\{ \sigma \in G_K : \sigma(P_i) = P_i \ \text{ for all } 1 \leq i \leq n \}.$	
\end{definition}

\begin{definition} \label{good outside S}
	Let $S$ be a finite set of primes in $\mathcal{O}_K$, where $\mathcal{O}_K$ denotes the ring of integers of $K$. 
	Let $X \subset \mathbb{P}^N(\overline{K})$ be a finite $\Gal(\overline{K}/K)$-invariant subset of cardinality at least $N+1$ with field of definition $K(X)$.  We say that $X$ has \emph{linear good reduction outside $S$} if for all primes $\mathfrak{p} \notin S$ and all primes $\mathfrak{P}$ in $K(X)$ lying over $\mathfrak{p}$, the reduction of $X$ modulo $\mathfrak{P}$ is in general linear position. In other words, the map that sends an $N$-element subset $Z\subset X$ to $H_Z\textup{ mod }\mathfrak{P}$ is injective for all such primes $\mathfrak{P}$ over $\mathfrak{p}$.
	\end{definition} 

\begin{remark}
When $N=1$, our definition of linear good reduction agrees with the standard notion of good reduction, which simply requires that the points remain pairwise distinct modulo all primes $\mathfrak{P}$ lying above $\mathfrak{p}$. 
\end{remark}

\begin{definition}\label{goodtriples} Let $d\geq 2$, $N\geq 1$, and $m\geq 1$ be integers. Let $K$ be a number field and $S$ a finite set of places of $K$ containing the archimedean places. We define the set of triples \[\mathcal{R}_{d,N}[m](K,S)=\{(f, X, Y) \text{ satisfying the following properties (1)-(6)}\}\]
\begin{enumerate}
	\item $f:\mathbb{P}^N\rightarrow \mathbb{P}^N$ is a degree $d$ morphism defined over $K$ with good reduction outside $S$,
	\item $X$ is a $\Gal(\Kbar/K)$-invariant subset of $\mathbb{P}^N(\overline{K})$,
	\item $X$ has linear good reduction outside $S$,
	\item $Y\subset \mathbb{P}^N(\Kbar)$ satisfies $|Y| = m$,
	\item $X = Y \cup f(Y)$,	
	\item \label{finalcondition} $Y$ is not contained in a hypersurface of degree at most $2d$.
\end{enumerate}
\end{definition}

\begin{remark}\label{rmk:N+1} By Lemma \ref{lem:generic}, condition (\ref{finalcondition}) in Definition \ref{goodtriples} can only be satisfied if $|Y|\ge\binom{N+2d}{2d}$. This forces $|Y|>N+1$, and hence $|X|>N+1$, so that our Definition \ref{good outside S} for linear good reduction is applicable to the set $X$. \end{remark}

It is not difficult to show that there exist configurations of $k$ points in $\mathbb{P}^N(\overline{K})$ satisfying the final condition provided $k\ge\binom{N+2d}{2d}$. (We prove this in the Appendix.) In fact, it is the generic expectation that a subset of at least $\binom{N+2d}{2d}$ distinct points is not contained in any degree $2d$ hypersurface, so this can be thought of as an additional general position condition. Further, as shown in the Appendix, any set $Y$ satisfying the final condition must have cardinality at least $\binom{N+2d}{2d}$.  

Throughout this paper we write, for a ring $R$, \[\textup{PL}_{N+1}(R):=\textup{GL}_{N+1}(R)/R^\times\hookrightarrow\textup{PGL}_{N+1}(R).\] When $\textup{Pic}(R)$ is trivial, $\textup{PL}_{N+1}(R)=\textup{PGL}_{N+1}(R)$ \cite[Corollary 2.7]{Faber}. Given $\phi \in\textup{PL}_{N+1}(\mathcal{O}_S)$ and $(f, X, Y)\in\mathcal{R}_{d,N}[m](K,S)$, define 
$$(f, X, Y)^{\phi} := (\phi^{-1} \circ f \circ \phi, \phi^{-1}(X), \phi^{-1}(Y)).$$ The map $f$ has good reduction outside $S$ if and only if $\phi^{-1}\circ f\circ\phi$ has good reduction outside $S$ because every $\phi\in \textup{PL}_{N+1}(\mathcal{O}_S)$ induces a scheme-theoretic isomorphism $\mathbb{P}^N_{\mathcal{O}_\mathfrak{p}}\rightarrow \mathbb{P}^N_{\mathcal{O}_\mathfrak{p}}$ for every $\mathfrak{p}\notin S$. That $X$ has linear good reduction outside $S$ if and only if $\phi^{-1}(X)$ has good reduction outside $S$ follows from the fact that $\phi\in\textup{PL}_{N+1}(\mathcal{O}_S)$ induces an automorphism of $\mathbb{P}^N(k_\mathfrak{p})$ for the residue field $k_\mathfrak{p}$ of each $\mathfrak{p}\notin S$. We thus obtain an action of $\textup{PL}_{N+1}(\mathcal{O}_S)$ on $\mathcal{R}_{d,N}[m](K,S)$. 

In this paper, we prove the following Shafarevich-style finiteness theorem.

\begin{theorem}\label{thm:dynShaf} Let $K$ be a number field, let $S$ be a finite set of places of $K$ containing the archimedean ones, let $d\ge 2$ and $N, m\ge 1$ be integers, and let $\mathcal{R}_{d,N}[m](K,S)$ be as in Definition \ref{goodtriples}. Then  $\mathcal{R}_{d,N}[m](K,S)$ admits only finitely many $\textup{PL}_{N+1}(\mathcal{O}_S)$-orbits.
\end{theorem}

\begin{remark}
Theorem \ref{thm:dynShaf} in the case $N=1$ implies Shafarevich's theorem (namely, Theorem \ref{thm:Faltings} in the case of elliptic curves). One easily verifies this by taking $f$ to be the Latt\`es map corresponding to the duplication on a Weierstrass representative of the elliptic curve and letting $Y\subset\mathbb{P}^1$ correspond to the union of the $2$-torsion and the $4$-torsion. (This is reminiscent of an argument used by Szpiro--Tucker \cite[\S2]{SzpiroTucker}.) Therefore, a natural question emerges: how does Theorem \ref{thm:dynShaf} relate to Faltings' Theorem \ref{thm:Faltings} in the setting of higher-dimensional abelian varieties? We leave this question to the reader.
\end{remark}

Theorem \ref{thm:dynShaf} in the case $N=1$ recovers \cite[Theorem 2]{Silverman:Shafarevich} with the restriction that we consider only triples $(f,X,Y)$ for which $f$ has multiplicity $1$ on $Y$. 
As in \cite{Silverman:Shafarevich}, however, our proof allows multiplicities to be taken into account when $N=1$. As Silverman explains in \cite[pp.~152-153]{Silverman:Shafarevich}, this result for $N=1$ implies the Shafarevich-style results due to Szpiro, Tucker, and West in \cite[Theorem 1]{SzpiroTucker} and \cite[Theorem 1.8]{SzpiroWest}. Each of these latter results in turn also imply Shafarevich's theorem (namely, Theorem \ref{thm:Faltings} in the case of elliptic curves), by the argument given in \cite[\S 2]{SzpiroTucker}. 

We remark that condition (\ref{finalcondition}) defining $\mathcal{R}_{d,N}(K,S)$ plays an analogous role in our proof as Silverman's condition that $\sum_{P\in Y}e_f(P)\ge 2d+1$ in his proof for the case $N=1$, where $e_f(P)$ is the local degree of $f$ at $P$ \cite[Theorem 2]{Silverman:Shafarevich}. Moreover, we claim that (\ref{finalcondition}) is in fact the ``correct" generalization of Silverman's condition, as Theorem \ref{thm:dynShaf} may be shown to fail in the absence of (\ref{finalcondition}). This failure can occur for $Y$ of arbitrarily large cardinality. The following example demonstrates this phenomenon.

\begin{example}
Let $K=\mathbb{Q}$, let \[f_c=[x_0^2:c(x_0^2+x_1^2-x_2^2)+x_0^2+x_1^2:x_0^2+x_1^2-x_2^2]\] for $c\in\mathbb{Z}$, and let $Y\subset\mathbb{P}^2(\mathbb{Q})$ be any set of cardinality at least $4$ in general linear position which is contained in the hypersurface $Z:x_0^2+x_1^2-x_2^2=0$.  Let $S$ be a finite set of places of $\mathbb{Q}$ containing the archimedean place such that $Y\cup f_c(Y)$ has  linear good reduction outside of $S$. (Note that $f_c(Y)$ is independent of $c$, and that we can form such a set $Y$ by taking suitable Pythagorean triples.) Since $|Y|\ge4$ and $Y$ is in general linear position, there are finitely many $\phi\in\textup{PGL}_3(\overline{K})$ leaving $X:=Y\cup f_c(Y)$ invariant. It follows that there are infinitely many $\textup{PGL}_3(\overline{K})$-conjugacy classes of triples in $\{(f_c,X,Y):c\in\mathbb{Z}\}$, and a fortiori infinitely many $\textup{PL}_3(\mathcal{O}_S)$-conjugacy classes of such triples. Further, $f_c$ has good reduction at all non-archimedean places for all $c\in\mathbb{Z}$. If the final condition defining $\mathcal{R}_{d,N}[m](K,S)$ were dropped, each of these triples would be an element of $\mathcal{R}_{2,2}[m](\mathbb{Q},S)$, where $m=|Y|$. 
\end{example} Similar examples may be constructed for any $N\ge 2$. We thus see that in the absence of the final condition, there is no $M$ such that Shafarevich finiteness holds for all $K,S$ and $m\ge M$ as stated in Theorem \ref{thm:dynShaf}.

The general strategy of our proof is similar to that of \cite{Silverman:Shafarevich}. First, we use the Hermite--Minkowski theorem to reduce the problem to sets $X$ with a fixed splitting field $L$. This is the content of \S\ref{reduction to finiteness over L}. Next, we show that large sets $X\subset\mathbb{P}^N(L)$ having linear good reduction outside $S$ yield many solutions to the classical $S$-unit equation over $L$. The choice of geometry behind our definition of linear good reduction is crucial to establish this connection between sets of points with linear good reduction and the $S$-unit equation. The fact, due to Siegel and Mahler, that the $S$-unit equation has only finitely many solutions allows us to show that, given $K$, $S$ and $n\gg1$, and up to $\text{PL}_{N+1}(\mathcal{O}_S)$-conjugacy, there are only finitely many $n$-element Galois-invariant $X\subset\mathbb{P}^N(\overline{K})$ having linear good reduction outside of $S$. Finally, we show that given such an $X$ and $d\ge2$, there are only finitely many morphisms $f:\mathbb{P}^N\to\mathbb{P}^N$ of degree $d$ defined over $K$ that realize $X$ as $Y\cup f(Y)$ for some $Y$ not contained in a hypersurface of degree at most $2d$. 

Each of the elements of our argument, specialized to the case $N=1$, essentially recovers Silverman's proof in \cite{Silverman:Shafarevich}. The primary difference in the higher-dimensional setting lies in the need to accommodate the fact that the action of $\textup{PGL}_{N+1}$ on $\mathbb{P}^N$ is only generically as opposed to fully $(N+1)$-transitive. It is also notable that our proof requires neither Schmidt's subspace theorem nor the generalized $S$-unit equation \cite[Theorem 7.4.1]{bombierigubler} with more than two terms, as one might na\"{i}vely expect given that these are the most natural higher-dimensional analogues of the finiteness of the set of solutions to the $S$-unit equation.

\textbf{Acknowledgements}: The authors began this work at the Women in Numbers Europe 3, and would like to thank the organizers of this workshop. They would also like to thank Xander Faber for helpful discussions on the automorphism group of projective space over general rings. During the preparation of this paper, HK was supported by the Radcliffe Institute as the 2021-22 Sally Starling Seaver Fellow, NL was partially supported by NSF grant DMS-1803021 and NMM was partially supported by NSF grant DMS-2200981. Finally, the authors extend their sincere gratitude to the referees for a very careful reading of this paper and many insightful comments.

\section{Reduction to finiteness for points defined over a single field extension}\label{reduction to finiteness over L}

Inspired by \cite[Sublemma 8]{Silverman:Shafarevich}, in this section we first prove the non-dynamical statement that the Galois-invariant sets $X\subset\mathbb{P}^N$ of a given cardinality having linear good reduction outside $S$ have all of their points defined over a fixed finite extension of $L/K$. 
This is done in Lemma \ref{one extension for all}, which is then used to deduce Lemma \ref{lem:finitetoone}. The finite-to-oneness of the map (\ref{eqn:finitetoone}) appearing in Lemma \ref{lem:finitetoone} in turn allows us to reduce Theorem \ref{thm:dynShaf} to the case where $K$ is replaced with $L$, $S$ is replaced by the set $T$ of primes of $L$ lying above $S$, and $X$ is contained in $\mathbb{P}^N(L)$ (as opposed to $X\subset\mathbb{P}^N(\overline{L})$ simply being $\Gal(\overline{L}/L)$-invariant). 

\begin{lemma}\label{one extension for all} Fix integers $d\geq 2$, $N\geq 1$, and $m\geq 1$. Then for all number fields $K$ and finite set of places $S$, there exists a fixed finite extension $L$ of $K$ such that for any $(f, X, Y) \in \mathcal{R}_{d,N}[m](K, S)$, each point of $X$ is defined over $L$.
\end{lemma}

\begin{proof}

Let $K(X)$ be the field of definition of the point set $X$ over $K$ as in Definition \ref{def:fieldofdefinition}.
As $X$ is $G_K$-invariant, the extension degree $[K(X) : K]$ is bounded above by a constant depending only on $|X|$ and $N$, where $m=|Y|\leq |X| \leq 2|Y|=2m$. 
	
If $K(X) / K$ ramifies at a prime $\mathfrak{p}$, then the inertia group is non-trivial, so there is a non-trivial element $\sigma$ of $G_K$ which preserves the valuation associated to a prime $\mathfrak{P}$ lying over $\mathfrak{p}$.  Since the element $\sigma$ is non-trivial and $K(X)$ is the fixed field of all elements acting trivially on the points of $X$, there is some $P_i \ne P_j$ such that $\sigma(P_i) = P_j$.  By definition of the inertia group, $P_i \equiv P_j \textup{ mod } \mathfrak{P}$, which implies that $X$ does not have linear good reduction at $\mathfrak{p}$. Thus, we conclude that $K(X)/K$ is unramified outside $S$.
	
The Hermite--Minkowski Theorem \cite[Theorem III.2.13]{MR1697859} states that for an algebraic number field $K$ and a finite set of primes $S$ of $K$, there exist only finitely many extensions $L/K$ of a given degree $n$ which are unramified outside $S$. Thus, there are only finitely many such $K(X)$, and we may take their compositum to assume each point of $X$ is defined over a single field as claimed.	
\end{proof}	

\begin{definition} We define a collection of subsets of $\mP^N(\overline{K})$ by
\[\mathcal{X}[n](K,S)= \left\{X\subset \mP^N(\overline{K}) : 
	\begin{array}{l}
	\text {$|X|=n$, $X$ is $\Gal(\overline{K}/K)$-invariant,}\\
	\text{and $X$ has linear good reduction outside $S$.}
	\end{array}\right\}\]
\end{definition}

\begin{definition}\label{def:setofpoints}
  Fix the following $N+2$ points in $\mP^N$;
\begin{align*}P_0=[1:0:\cdots:0],  P_1&=[0:1:0:\cdots:0],\dots,\\
P_N &=[0:\cdots:0:1], P_{N+1}=[1:\cdots:1].
\end{align*}
\end{definition}

\begin{lemma}\label{lem:lineartransformation} Let $L$ be a number field and $T$ a finite set of primes of $L$ containing all archimedean places. Suppose $\{Q_0, \dots, Q_{N+1}\}\subset \mP^N(L)$ has linear good reduction outside of $T$. Then  there is a unique linear transformation $\psi\in\textup{PL}_{N+1}(\mathcal{O}_T)$ such that 
	\[\psi(P_{i})=Q_{i} \text{ for all } 0\leq i\leq N+1,\]
where $P_0,\dots,P_{N+1}$ are the points in Definition \ref{def:setofpoints}.
\end{lemma}

\begin{proof} For a prime $\mathfrak{P}\notin T$, let $R_\mathfrak{P}$ denote the valuation ring of $K_\mathfrak{P}$, so that $\mathcal{O}_T=\bigcap_{\mathfrak{P}\notin T}R_\mathfrak{P}$. Fix an arbitrary $\mathfrak{P}\notin T$. For each $Q_i$, we choose a representative $\hat{Q}_i=(x_{i,0},\dots, x_{i,N})\in \mathbb{A}^{N+1}(L)$ for $Q_i$ such that each $x_{i,j}\in R_\mathfrak{P}$ and at least one $x_{i,j}$ is a $\mathfrak{P}$-adic unit, and similarly for $P_i$. 

Since the $Q_i$ are in general linear position, any $N+1$ of them span $\mathbb{A}^{N+1}(L)$. Thus there exist unique $\lambda_0,\dots,\lambda_N\in L$ such that \[\hat{Q}_{N+1}=\sum_{i=0}^{N}\lambda_i\hat{Q}_i.\] We note that in fact, $\lambda_i\in L^*$ for all $i$ since if some $\lambda_i$ were equal to $0$, then the $Q_j$ would not be in general linear position, because $Q_{N+1}$ would be in the linear subspace spanned by $Q_0,\dots,Q_{i-1},Q_{i+1},\dots,Q_{N}$. Similarly, note that the linear good reduction assumption implies that the $\lambda_i$ lie in $R_\mathfrak{P}^{\times}$, and so the $\lambda_i\hat{Q}_i$ are in $R_\mathfrak{P}$.
Rescaling the $\hat{Q}_i$ gives \[\hat{Q}_{N+1}=\sum_{i=0}^{N}\hat{Q}_i.\] Taking $M\in\textup{GL}_{N+1}(L)$ to have columns $\hat{Q}_i$ gives $M\hat{P}_i=\hat{Q}_i$ for all $i$. 
By construction, $M$ has $R_\mathfrak{P}$ entries, and the reduction of $M$ modulo $\mathfrak{P}$ yields an element of $\textup{GL}_{N+1}(k_\mathfrak{P})$. Thus $M\in\textup{GL}_{N+1}(R_\mathfrak{P})$. Since $\mathcal{O}_T=\bigcap_{\mathfrak{P}\notin T}R_\mathfrak{P}$ and the preceding argument holds for arbitrary $\mathfrak{P}\notin T$, it follows that $M\in\textup{GL}_{N+1}(\mathcal{O}_T)$. 
The corresponding transformation $\psi\in\textup{PL}_{N+1}(\mathcal{O}_T)$ we are after is the image of $M$ in $\textup{GL}_{N+1}(\mathcal{O}_T)/\mathcal{O}_T^\times$. 
\end{proof}

\begin{lemma}\label{lem:finitetoone} Let $N\ge1$, let $n\geq N+2$, let $K$ be a number field, and let $S$ be a finite set of places of $K$ containing all archimedean places. Then the natural map \[\mathcal{X}[n](K,S)/\textup{PL}_{N+1}(\mc{O}_S)\rightarrow \{X\subset \mP^N(\overline{K}):|X|=n\}/\textup{PGL}_{N+1}(\overline{K})\}\] is finite-to-one. In particular, for $L$ and $T$ as in Lemma \ref{one extension for all}, the natural map \begin{equation}\label{eqn:finitetoone}\mathcal{X}[n](K,S)/\textup{PL}_{N+1}(\mc{O}_S)\rightarrow\mathcal{X}[n](L,T)/\textup{PL}_{N+1}(\mc{O}_T)\end{equation} is finite-to-one.\end{lemma}

\begin{proof} We generalize Silverman's proof of \cite[Sublemma 8]{Silverman:Shafarevich}. Fix $X_0 \in \mathcal{X}[n](K,S)$ and let \[\textup{PGL}_{N+1}(K,S,X_0) = \{\phi\in\textup{PGL}_{N+1}(\overline{K}) : \phi(X_0)\in \mathcal{X}[n](K,S)\}.\] 
 Note that $\textup{PL}_{N+1}(\mathcal{O}_S)\subset \textup{PGL}_{N+1}(K,S,X_0)$.
 It suffices to define a map from $\textup{PGL}_{N+1}(K,S,X_0)$ to a finite set and to show that if two elements $\phi_1, \phi_2 \in\textup{PGL}_{N+1}(K,S,X_0)$ have the same image under this map, then they differ by an element of $\textup{PL}_{N+1}(\mc{O}_S)$. 
This then proves that $\textup{PGL}_{N+1}(K,S,X_0)/\textup{PL}_{N+1}(\mc{O}_S)$ is finite.
	
Let $L$ be as in Lemma \ref{one extension for all}. Let $T$ be a finite set of places of $L$ such that the restriction of $T$ to $K$ is $S$. We first claim that if $\phi\in\textup{PGL}_{N+1}(K,S,X_0)$, then $\phi \in\textup{PL}_{N+1}(\mc{O}_T)$. From Lemma \ref{one extension for all}, we have $X_0, \phi(X_0)\subset \mP^N(L)$. Any element of $\textup{PGL}_{N+1}(\overline{K})$ is determined by its values on $N+2$ points such that no $N+1$ of them are contained in a hyperplane, so since $n\geq N+2$, we have $\phi\in\textup{PGL}_{N+1}(L)$. 
	
	Moreover, the linear good reduction assumptions on $X_0$ and $\phi(X_0)$ imply that for any prime $\mathfrak{P}$ of $L$ with $\mathfrak{P}\notin T$, any $N+1$ points in the reduction of $X_0$ mod $\mathfrak{P}$ are not contained in a hyperplane, and similarly for $\phi(X_0)$. We claim that $\phi\in \textup{PL}_{N+1}(\mc{O}_T)$. To see this, let $Q_0,\dots, Q_{N+1}$ be $N+2$ points in $X_0$ and let $Q_0',\dots, Q_{N+1}'$ satisfy $\phi(Q_i)=Q_i'$ for all $i$.

By Lemma \ref{lem:lineartransformation}, there is a linear transformation $\psi\in\textup{PL}_{N+1}(\mathcal{O}_T)$ such that \[\psi(P_i)=Q_i \text{ for all } 0\leq i\leq N+1.\]  Similarly, there is a linear transformation $\lambda \in \textup{PL}_{N+1}(\mathcal{O}_T)$ with \[\lambda(P_i)=Q_i' \text{ for all } 0\leq i\leq N+1.\] As noted above, $\phi$ is completely determined by its action on $Q_0,\dots, Q_{N+1}$. Hence, \[\phi=\lambda\circ\psi^{-1} \in\textup{PL}_{N+1}(\mc{O}_T).\]
	
Next we let $S_{X_0}$ denote the set of permutations of $X_0$, and define a map 
\begin{align*}
\mathcal{M}: \textup{PGL}_{N+1}(K,S,X_0)&\rightarrow \textup{Map}_{\textup{Set}}(\Gal(L/K), S_{X_0}) \\ 
\phi&\mapsto (\sigma\mapsto \phi^{-1}\circ \phi^\sigma).
\end{align*} 
Note that $\Gal(L/K)$ and $ S_{X_0}$ are both finite, so $\textup{Map}_{\textup{Set}}(\Gal(L/K), S_{X_0})$ is finite. 
To see that this map is well-defined, we must check that for any $\phi\in\textup{PGL}_{N+1}(K,S,X_0)$ and any $\sigma \in \Gal(L/K)$, we have $\phi^{-1}\circ \phi^\sigma: X_0\rightarrow X_0$. Let $\phi \in\textup{PGL}_{N+1}(K,S,X_0)$ and $\sigma \in \Gal(L/K)$. The sets $X_0, \phi(X_0)\subset \mP^N(L)$ are both $\Gal(L/K)$-invariant. Hence, for all $\sigma\in \Gal(L/K)$, we have \[\phi(X_0) = (\phi(X_0))^\sigma=\phi^\sigma(X_0^\sigma)=\phi^\sigma(X_0).\] This yields $\phi^{-1}\circ \phi^\sigma:X_0\rightarrow X_0$.

Now fix $\phi_1,\phi_2 \in\textup{PGL}_{N+1}(K,S,X_0)$ and suppose that 
$$\mathcal{M}(\phi_1)=\mathcal{M}(\phi_2) \in\textup{Map}_{\textup{Set}}(\Gal(L/K), S_{X_0}).$$ Then for all $\sigma \in \Gal(L/K)$, $\phi_1^{-1}\circ \phi_1^\sigma$ and $\phi_2^{-1}\circ\phi_2^\sigma$ have the same action on $X_0$. Since $X_0$ has $n\geq N+2$ points in general position, we must have  $\phi_1^{-1}\circ \phi_1^\sigma = \phi_2^{-1}\circ\phi_2^\sigma$ as elements of $\textup{PGL}_{N+1}(L)$. So \begin{equation}\label{eqn:sigmaidentity}\phi_1\circ \phi_2^{-1} = \phi_1^\sigma\circ(\phi_2^\sigma)^{-1} = (\phi_1\circ\phi_2^{-1})^\sigma.\end{equation} Since this holds for all $\sigma \in \Gal(L/K)$, we conclude that $\phi_1\circ\phi_2^{-1} \in\textup{PGL}_{N+1}(K)$; indeed, for $\Upsilon\in\textup{GL}_{N+1}(L)$ any representative of $\phi_1\circ\phi_2^{-1}$, the relation (\ref{eqn:sigmaidentity}) implies that each $\sigma\in\Gal(L/K)$ fixes all ratios of entries in $\Upsilon$, and hence that all of these ratios lie in $K$.
We have already seen that $\phi_1, \phi_2\in\textup{PL}_{N+1}(\mc{O}_T)$, and thus we conclude that $\phi_1\circ\phi_2^{-1}\in\textup{PL}_{N+1}(\mc{O}_{S})$.\end{proof}

\section{Reduction to bounds on the number of solutions to $S$-unit equations}

In this section, we prove the finiteness (up to $\text{PL}_{N+1}(\mathcal{O}_S)$-conjugacy) of the number of sets $X\subset\mathbb{P}^N(K)$ having linear good reduction outside $S$, with the number of such sets depending only on $|X|$ and $|S|$; see Corollary \ref{cor:Xgoodred}. We will show that this comes down to the finiteness of the number of solutions to the classical $S$-unit equation.

\begin{theorem}[$S$-unit Theorem]\cite[Theorem 5.2.1]{bombierigubler}\label{thm:s-unit} There are absolute computable constants $C_1, C_2$ with the following property. Let $\Gamma$ be a subgroup of $\overline{\mathbb{Q}}\times \overline{\mathbb{Q}}$ with $\operatorname{rank}_\mathbb{Q}(\Gamma)=r<\infty$, where  $\operatorname{rank}_\mathbb{Q}(\Gamma)$ is the maximum number of multiplicatively independent elements in $\Gamma$. Then the equation \[x+y=1, (x,y)\in \Gamma\] has at most $C_1\cdot C_2^r$ solutions.
\end{theorem}

Recall that in Definition \ref{def:setofpoints}, we specified $N+2$ points $P_0,\dots,P_{N+1}$. 
We denote the homogeneous coordinates of $\mathbb{P}^N$ by $z_0,\ldots,z_N$ and let
\begin{equation}\label{eqn:Hexpression} \mathcal{H}:=\left(\displaystyle\bigcup_{\substack{0\le i\neq j\le N}}\{z_i-z_j=0\}\right)\cup\left(\displaystyle\bigcup_{i=0}^N\{z_i=0\}\right)\end{equation} be the union of the big diagonal and the coordinate hyperplanes.
Note that $P_0,\ldots,P_N\in\mathcal{H}$.  

\bigskip

In the rest of the section we make the following assumption. \\

\noindent\textbf{Assumption PID: }$S$ is large enough so that $\mathcal{O}_S$ is a PID. \\

\begin{remark}\label{pid}
    By the finiteness of the ideal class group of $K$, it follows that for any given set $S'$ of places of $K$ there is a finite set $S''$ of places of $K$ so that $S'\cup S''$ satisfies Assumption PID. 
\end{remark}

\begin{definition}\label{normalized}
    A point $x\in \mathbb{P}^N(K)$ is in $S$-normalized form if it is written in homogeneous coordinates $x=[x_0:x_1:\cdots:x_n]$ satisfying
    \begin{itemize}
	\item $x_0,\ldots,x_N\in \mathcal{O}_S$; and 
	\item $\max\{|x_0|_{\mathfrak{p}},\ldots,|x_N|_{\mathfrak{p}}\}=1$ for all primes $\mathfrak{p}\notin S$.\end{itemize}
\end{definition}

 \begin{remark}
     Since $S$ satisfies Assumption PID, every point $x\in\mathbb{P}^N(K)$ can be written in $S$-normalized form.
 \end{remark}

The following lemma will be crucial in the proof of Theorem \ref{thm:dynShaf}.
\begin{lemma}\label{tounit}
	Let $N\ge 1$, let $K$ be a number field and let $S$ be a finite set of primes of $K$, with group of $S$-units denoted $\mathcal{O}^{*}_S$.
	Let $X\subset \mathbb{P}^N(K)$ be such that
	\begin{enumerate}
		\item 
		$X$ has linear good reduction outside $S$.
		\item 
		$P_0,\ldots,P_N,P_{N+1}\in X$.
	\end{enumerate}
	Then, for all $x=[x_0:x_1:\cdots: x_N]\in X\setminus \mathcal{H}$ written in $S$-normalized form, we have 
 \begin{align*}
  x_i-x_j&\in \mathcal{O}^{*}_S,~\text{for all }0\le i<j\le N,\\
  x_i&\in \mathcal{O}^{*}_S,~~\text{for all }0\le i\le N.
 \end{align*}
\end{lemma}

\begin{proof}
	Let $\mathfrak{p}$ be a prime of $K$ not in $S$ and let $x=[x_0:x_1:\cdots:x_n]\in X\setminus \mathcal{H}$ be written in $S$-normalized form. 
	For $i\neq j$, let $H_{ij}$ be the hyperplane $z_i=z_j$. We note that $x\notin H_{ij}$ but that $H_{ij}$ contains $N$ of the $P_k$; more precisely 
 \begin{align*}
     \{P_k~:~0\le k\le N, k\neq i, k\neq j\}\subset H_{ij}.
 \end{align*}
	Since $X$ has linear good reduction outside $S$, we thus see that $x\textup{ mod } \mathfrak{p}$ cannot lie in the hyperplane $H_{ij} \textup{ mod } \mathfrak{p}$. 
	Therefore $x_i\neq x_j \textup{ mod } \mathfrak{p}$. 
	Recalling that $x_0,x_1,\ldots,x_N$ are $S$-integers and $\mathfrak{p}\notin S$ is arbitrary we get that $x_i-x_j\in \mathcal{O}^{*}_S$ as claimed. Similarly, for each $i$, let $H_i$ be the hyperplane $z_i=0$. Then $x\notin H_i$, but $H_i$ contains $N$ of the points $P_k$. Since $X$ has linear good reduction outside $S$, we infer that $x_i\in \mathcal{O}^{*}_S$ for all $i$. 
\end{proof}

\begin{corollary}\label{cor:Xgoodred}
	Let $N\ge1$, let $K$ be a number field and let $S$ a finite set of primes of $K$. 
	There is a finite set $\Pi=\Pi(K,S,N)\subset\mathbb{P}^N(K)$ such 
	if $X\subset \mathbb{P}^N(K)$ satisfies:
	\begin{enumerate}
		\item 
		$X$ has linear good reduction outside $S$,
		\item 
		$P_0,\ldots,P_N,P_{N+1}\in X$,
	\end{enumerate}
	then $X\subseteq \Pi$.
\end{corollary}

\begin{proof} 
Let $\mathcal{H}$ be the set of hyperplanes defined by \eqref{eqn:Hexpression}. 
First note that if $x\in X\cap\mathcal{H}$, then $x\in\{P_0,P_1,\dots,P_{N+1}\}$, since otherwise, one of the hyperplanes in $\mathcal{H}$ contains $N+1$ points of $X$, contradicting the assumption that the points of $X$ are in general position. Thus we can suppose without loss of generality that $x=[x_0:\cdots:x_N]\in X\setminus\mathcal{H}$, written in $S$-normalized form as described in Definition \ref{normalized}. Let 
$$\Pi_0=\Pi_0(K,S)=\{u\in \mathcal{O}_S^{*}~:~1-u\in \mathcal{O}_S^{*}\}.$$
Then, the $S$-unit Theorem \ref{thm:s-unit} says that $\Pi_0$ is a finite set. We define 
$$\Pi=\Pi(K,S)=\{[1:u_1:\cdots: u_N]~:~u_1,\ldots,u_N\in\Pi_0\},$$
so $\Pi$ is also a finite set. 

Next, for each $1\le i \le N$, we know from Lemma \ref{tounit} that $x_0,-x_i$, and $x_0-x_i$ are each $S$-units, which implies that 
\begin{align*}
    \frac{x_i}{x_0}\in\mathcal{O}_S^{*}\text{ and }1-\frac{x_i}{x_0}=\frac{x_0-x_i}{x_0}\in\mathcal{O}_S^{*}.
\end{align*}
It follows that $x_i/x_0\in\Pi_0$ for every $1\le i\le N$. Since $x\in\mathbb{P}^N$, we conclude that 
\begin{align*}
    x=[x_0:x_1:\cdots:x_N]=\left[1:\frac{x_1}{x_0}:\cdots:\frac{x_N}{x_0}\right]\in \Pi.
\end{align*}
Since $x$ was an arbitrary element of $X$, this completes the proof that $X\subseteq \Pi$.
\end{proof}

\begin{remark} 
Although the proof of Corollary \ref{cor:Xgoodred} assumes that $S$ satisfies Assumption PID in order to be able to invoke Lemma \ref{tounit}, we can always adjoin finitely many additional primes to $S$ and ensure that Assumption PID is satisfied as explained in Remark \ref{pid}. Therefore, Corollary \ref{cor:Xgoodred} is true without Assumption PID. 
\end{remark}

\section{Proof of Shafarevich finiteness}

\begin{proof}[Proof of Theorem \ref{thm:dynShaf}] Let $(f,X,Y)\in\mathcal{R}_{d,N}[m](K,S)$. By Lemma \ref{one extension for all}, there is a fixed extension $L/K$ depending only on $S$ such that $X\subset\mathbb{P}^N(L)$. Let $T$ be the finite set of primes of $L$ lying above $S$.
	Since $X$ is in general linear position, conjugating $(f,X,Y)$ by an element of $\textup{PL}_{N+1}(\mathcal{O}_T)$ if necessary as in Lemma \ref{lem:lineartransformation}, we may assume that $P_0,\dots,P_{N+1}\in X$. Here we have used the fact that $|X|\ge N+2$; see Remark \ref{rmk:N+1}. 
	 By Corollary \ref{cor:Xgoodred}, it follows that $X$ is contained in a finite set $\Pi'=\Pi'(L,T,N)$ of points in $\mathbb{P}^N(L)$. On the other hand, for any degree $d$ morphism $g:\mathbb{P}^N\to\mathbb{P}^N$ distinct from $f$, we claim that the subvariety given by $f=g$ is contained in a hypersurface of degree $2d$. Indeed, writing $f=[f_0:\cdots:f_N]$ and $g=[g_0:\cdots:g_N]$, this is the set \[\bigcap_{0\le i<j\le N}\{P\in\mathbb{P}^N:f_i(P)g_j(P)=f_j(P)g_i(P)\}.\] Thus, choosing $i,j$ such that $f_ig_j\ne f_jg_i$ as polynomials, we have that \[\{f=g\}\subseteq\{f_ig_j=f_jg_i\},\] which proves the claim. Since by assumption $Y$ is not contained in any degree $2d$ hypersurface, we see that $f$ is uniquely determined by its action on $Y$. Furthermore, as $Y \subseteq \Pi'$, there are only finitely many possibilities for the set $Y$. It follows that $\mathcal{R}_{d,N}[m](L,T)$ consists of only finitely many $\textup{PL}_{N+1}(\mathcal{O}_T)$-orbits. This implies, by Lemma \ref{lem:finitetoone}, that $\mathcal{R}_{d,N}[m](K,S)$ consists of only finitely many $\textup{PL}_{N+1}(\mathcal{O}_S)$-orbits.\end{proof}

\begin{appendix} 
	
	\section{}
	The purpose of the following lemma is to show that in our definition (Definition \ref{goodtriples}) of the set $\mathcal{R}_{d,N}[m](K,S)$ of good triples $(f,X,Y)$, the assumption on $Y$ not being contained in any degree $2d$ hypersurface is generically satisfied, provided $|Y|\ge\binom{N+2d}{2d}$.
	
	\begin{lemma}\label{lem:generic} Let $M,N,D\ge1$ be integers, and let $K$ be a field. Define \[\mathcal{Y}_{D,N}[M]=\left\{(P_1,\dots,P_M)\in(\mathbb{P}_K^N)^M: 
	\begin{array}{l}\textup{$\{P_1,\dots,P_M\}$ is not contained in } \\ \textup{a hypersurface of degree at most $D$.}
	\end{array}\right\}\]
 
\begin{enumerate}[(a)] \item\label{eqn:parta} If $M\ge\binom{N+D}{D}$, then $\mathcal{Y}_{D,N}[M]$ is a nonempty Zariski open subset of $(\mathbb{P}_K^N)^M$. \item\label{eqn:partb} If $M\le\binom{N+D}{D}-1$, then $\mathcal{Y}_{D,N}[M]=\emptyset$.\end{enumerate}\end{lemma} 

\begin{proof} Let $P_1,\dots,P_M$ be arbitrary points in $\mathbb{P}^N$, and let $F\in K[x_0,\dots,x_N]$ be a homogeneous polynomial defining a (hypothetical) degree $D$ hypersurface in $\mathbb{P}^N$ that contains $P_1,\dots,P_M$. Let $s_1,\dots,s_r$ be the monic degree $D$ monomials in the $N+1$ variables $x_0,\dots,x_N$, and write \begin{equation}\label{eqn:Fshape} F=\sum_{i=1}^ra_is_i\end{equation} for $a_i\in K$. As $r$ is the number of monic monomials of degree $D$ in $N+1$ variables, we have that $r=\binom{N+D}{D}$. We have \begin{equation*}\begin{array}{c} a_1s_1(P_1)+a_2s_2(P_1)+\dots+a_rs_r(P_1)=0 \\ \vdots \\ a_1s_1(P_M)+a_2s_2(P_M)+\dots+a_rs_r(P_M)=0.\end{array}\end{equation*} Switching our perspective, we can view the $a_i$ for $1\le i\le r$ as the coordinates of a vector $(a_1,\dots,a_r)\in K^r$ that is a solution to the system of equations \begin{equation}\label{eqn:system}\begin{array}{c} A_1s_1(P_1)+A_2s_2(P_1)+\dots+A_rs_r(P_1)=0 \\ \vdots \\ A_1s_1(P_M)+A_2s_2(P_M)+\dots+A_rs_r(P_M)=0,\end{array}\end{equation} where the $A_i$ are seen as indeterminates. But as $M\le r-1$, the matrix $(s_j(P_i))$ corresponds to a linear transformation with a non-trivial kernel, so in this case there exists a nontrivial solution regardless of what $P_1,\dots,P_M$ are. Any nontrivial solution $(a_1,\dots,a_r)$ yields, via (\ref{eqn:Fshape}), a hypersurface in $\mathbb{P}^N$ containing $P_1,\dots,P_M$. This proves (\ref{eqn:partb}).

Now suppose that $p_1,\dots,p_M$ are arbitrary points in $\mathbb{P}^N$, where $M\ge r$. We now view the $P_i$ in \eqref{eqn:system} as indeterminates. For each $i$, substitute $P_i=p_i$ in \ref{eqn:system}, yielding a system of equations in the indeterminates $A_i$. Since $M\geq r$, we have that \begin{equation}\label{eqn:keyfactminors}\begin{array}{c} \text{(\ref{eqn:system}) with $P_i=p_i$ for all $i$ has a} \\ \text{nontrivial solution $(a_1,\dots,a_r)$} \end{array} \Longleftrightarrow \begin{array}{c} \text{ all $r\times r$ minors of  } \\ \text{$(s_j(p_i))$ vanish.}\end{array}\end{equation} Let $\mathscr{M}$ be the set of all (not necessarily proper) square submatrices of $(s_j(P_i))$. Since for each $T\in\mathscr{M}$, $\textup{det}(T)$ is a nonzero polynomial in $P_1,\dots,P_M$, it follows that the locus $\mathcal{Z}\subsetneq(\mathbb{P}^N)^M$ of $(p_1,\dots,p_M)$ such that $(s_j(P_i))_{P_i=p_i}$ has some square submatrix with determinant equal to $0$ is Zariski closed. On the other hand, by (\ref{eqn:Fshape}) and (\ref{eqn:keyfactminors}), $\mathcal{Z}$ corresponds exactly to the points $(p_1,\dots,p_M)\in(\mathbb{P}^N)^M$ such that $\{p_1,\dots,p_M\}$ is contained in a degree $D$ hypersurface in $\mathbb{P}^N$. This proves (\ref{eqn:parta}).\end{proof}
		\end{appendix}

\bibliographystyle{plain}

\bibliography{Shaf}

\end{document}